\documentclass[a4paper,12pt,reqno]{amsart}
\usepackage{amsfonts}
\usepackage{amsmath}
\usepackage{amssymb}
\usepackage[a4paper]{geometry}
\usepackage{mathrsfs}
\usepackage{csquotes}
\usepackage{enumitem}
\usepackage[colorlinks]{hyperref}
\renewcommand\eqref[1]{(\ref{#1})} 

\numberwithin{equation}{section}
\theoremstyle{plain}
\newtheorem{thm}{Theorem}[section]

\newtheorem{cor}[thm]{Corollary}
\newtheorem{lem}[thm]{Lemma}
 
\theoremstyle{definition}

\newtheorem{rem}[thm]{Remark}

\newcommand{\Rn}{\mathbb R^{n}}

\def\e[#1]{{\textrm{e}}^{#1}}

\def\Rn{{\mathbb R}^n}
\def\G{{\mathbb G}}

\def\L{\mathfrak{L}}

\begin{document}

   \title[Global well-posedness for a semilinear heat
equation on manifolds]
   {Existence and non-existence of global solutions for semilinear heat equations and inequalities on sub-Riemannian manifolds, and Fujita exponent on unimodular Lie groups}

\author[M. Ruzhansky]{Michael Ruzhansky}
\address{
  Michael Ruzhansky:
  \endgraf
  Department of Mathematics: Analysis,
Logic and Discrete Mathematics
  \endgraf
  Ghent University, Belgium
  \endgraf
  and
  \endgraf
  School of Mathematical Sciences
    \endgraf
    Queen Mary University of London
  \endgraf
  United Kingdom
  \endgraf
  {\it E-mail address} {\rm michael.ruzhansky@ugent.be}
  }
 \author[N. Yessirkegenov]{Nurgissa Yessirkegenov}
\address{
  Nurgissa Yessirkegenov:
  \endgraf
  Department of Mathematics: Analysis, Logic and Discrete Mathematics
  \endgraf
    Ghent University, Belgium
\endgraf
  and
  \endgraf
  Suleyman Demirel University
  \endgraf
  Kaskelen, Kazakhstan
  \endgraf
  and
  \endgraf
  Institute of Mathematics and Mathematical Modeling, Kazakhstan
    \endgraf
  {\it E-mail address} {\rm nurgissa.yessirkegenov@gmail.com}
  }
\dedicatory{Dedicated to Tokio Matsuyama on the occasion of his 60$^{th}$ birthday}

\thanks{This research is funded by the Science Committee of the Ministry of Education and Science of the Republic of Kazakhstan (Grant No. AP09058474) and by the FWO Odysseus 1 grant G.0H94.18N: Analysis and Partial Differential Equations. Michael Ruzhansky was supported by the EPSRC grant EP/R003025/2 and by the Methusalem programme of the Ghent University Special
Research Fund (BOF) (Grant number 01M01021).   }

     \keywords{Semilinear heat equation, differential inequality, sub-Riemannian manifold, unimodular Lie group, global well-posedness, global solution, sub-Laplacian.}
     \subjclass[2010]{35K58, 58J35, 35R45.}

     \begin{abstract} In this paper we study the global well-posedness of the following Cauchy problem on a sub-Riemannian manifold $M$:
\begin{equation*}
\left\{
                \begin{array}{ll}
                  u_{t}-\L_{M} u=f(u), \;x\in M, \;t>0,\\
                  u(0,x)=u_{0}(x), \;x\in M,\\
                                  \end{array}
              \right.
\end{equation*}
for $u_{0}\geq 0$, where $\L_{M}$ is a sub-Laplacian of $M$. In the case when $M$ is a connected unimodular Lie group $\G$, which has polynomial volume growth, we obtain a critical Fujita exponent, namely, we prove that all solutions of the Cauchy problem with $u_{0}\not\equiv 0$, blow up in finite time if and only if $1<p\leq p_{F}:=1+2/D$ when $f(u)\simeq u^{p}$, where $D$ is the global dimension of $\G$. In the case $1<p<p_{F}$ and when $f:[0,\infty)\to [0,\infty)$ is a locally integrable function such that $f(u)\geq K_{2}u^{p}$ for some $K_{2}>0$, we also show that the differential inequality
$$
u_{t}-\L_{M} u\geq f(u)
$$
does not admit any nontrivial distributional (a function $u\in L^{p}_{loc}(Q)$ which satisfies the differential inequality in $\mathcal{D}^{\prime}(Q)$) solution $u\geq 0$ in $Q:=(0,\infty)\times\G$. Furthermore, in the case when $\G$ has exponential volume growth and $f:[0,\infty)\to[0,\infty)$ is a continuous increasing function such that $f(u)\leq K_{1}u^{p}$ for some $K_{1}>0$, we prove that the Cauchy problem has a global, classical solution for $1<p<\infty$ and some positive $u_{0}\in L^{q}(\G)$ with $1\leq q<\infty$. Moreover, we also discuss all these results in more general settings of sub-Riemannian manifolds $M$.
     \end{abstract}

  \maketitle

  \tableofcontents

\section{Introduction and results}
\label{SEC:intro}
Let $M$ be a sub-Riemannian manifold with a smooth volume $\mu$. Recall that if $X_{1},\ldots,X_{k}$ is a local orthonormal frame, then the horizontal gradient on $M$ is defined as
$$\nabla\varphi=\sum_{i=1}^{k}X_{i}(\varphi)X_{i},$$
where $X_{i}(\varphi)$ is the Lie derivative of $\varphi$ in the direction of $X_{i}$. Denote by ${\rm div}_{\mu}X$ the divergence of a vector field $X$ with respect to a volume $\mu$, and it is defined by the identity $L_{X}\mu=({\rm div}_{\mu}X)\mu$, where $L_{X}$ is the Lie derivative with respect to $X$. The sub-Laplacian associated with the sub-Riemannian structure is defined as the divergence of the gradient, that is, $\L_{M}\varphi={\rm div}_{\mu}(\nabla\varphi)$, and it can be written in a local orthonormal frame $X_{1},\ldots,X_{k}$ as
\begin{equation}\label{sublap_subR}
\L_{M}=\sum_{i=1}^{k}\left(X_{i}^{2}+({\rm div}_{\mu}X_{i})X_{i}\right).
\end{equation}
Therefore, the sub-Laplacian $\L_{M}$ is the natural generalisation of the Laplace-Beltrami operator defined on a Riemannian manifold. Note that the sub-Laplacian $\L_{M}$ always can be expressed as the sum of squares of the elements of the orthonormal frame plus a first order term that depends on the choice of the volume $\mu$.

Let $h_{t}(x,y)$ denote the heat kernel for $u_{t}-\L_{M} u=0$ for $x,y\in M$ and $t>0$, that is, for every $y \in M$ the function $u(t,x):=p_{t}(x,y)$, $x\in M$, $t>0$ is a classical solution to the heat equation $u_{t}-\L_{M} u=0$ in $(0, \infty)\times M$. Its existence, smoothness, symmetry, and positivity are guaranteed by classical results, see for instance \cite{Str86} or \cite[Section 2.2]{BBN12} and \cite[Section 2.1]{BBN16} for more references.

Let us consider the following Cauchy problem:
\begin{equation}\label{Cauchy_intro_Riem}
\left\{
                \begin{array}{ll}
                  u_{t}-\L_{M} u=f(u), \;x\in M, \;t>0,\\
                  u(0,x)=u_{0}(x), \;x\in M,\\
                                  \end{array}
              \right.
\end{equation}
for $u_{0}\geq 0$. Recall that
$$e^{t \L_{M}}u_{0}(x)=\int_{M}h_{t}(x,y)u_{0}(y)d\mu_{y}, \quad x\in M, \quad t>0.$$

In this paper, we show a sufficient condition on the initial data which guarantees the existence of global solutions of \eqref{Cauchy_intro_Riem} on $M$:
\begin{thm}\label{lem2} Let $M$ be a sub-Riemannian manifold. Let $f:[0,\infty)\to[0,\infty)$ be a continuous increasing function such that $f(u)\leq K_{1}u^{p}$ for some positive constant $K_{1}>0$ with $1<p<\infty$. Let $0\leq u_{0}\in L^{q}(M)$ with $1\leq q<\infty$ and assume that
\begin{equation}\label{lem2_for1}
\int_{0}^{\infty}\|e^{s\L_{M}}u_{0}\|_{L^{\infty}(M)}^{p-1}ds<\frac{1}{K_{1}(p-1)}.
\end{equation}
Then there exists a non-negative continuous curve $u:[0,\infty)\rightarrow L^{q}(M)$ which is a global solution to \eqref{Cauchy_intro_Riem} with initial value $u_{0}$. Moreover, we have
\begin{equation}\label{lem2_for2}
(e^{t\L_{M}}u_{0})(x)\leq u(t,x)\leq C(e^{t\L_{M}}u_{0})(x), \;\forall x\in M,\;\forall t\geq0,
\end{equation}
for some $C>1$ (depending on $u_{0}$). For example, \eqref{lem2_for2} holds with
$$C=\left(1-K_{1}(p-1)\int_{0}^{\infty}\|e^{s\L_{M}}u_{0}\|^{p-1}_{L^{\infty}(M)}ds\right)^{-\frac{1}{p-1}}.$$
\end{thm}
\begin{rem} We refer to \cite{Zha99} and the references therein for the case when $\L_{M}$ is the Laplace-Beltrami operator, $M$ is a noncompact, complete Riemannian manifold with polynomial volume growth, which include those with nonnegative Ricci curvatures. For the case of Riemannian manifolds with negative sectional curvature, we refer to \cite{Pun12} as well as references therein.
\end{rem}
When we know the behavior of the heat kernel, then one can see that to satisfy the condition \eqref{lem2_for1} there appears a condition for the parameter $p$, which is usually called a Fujita exponent. For example, when $M$ is a connected unimodular Lie group with polynomial volume growth of order $D$, since we have by \eqref{est_heat_above} the estimate
$$\|e^{t\L}u_{0}\|_{L^{\infty}(\G)}\geq ct^{-D/2},\;t\rightarrow\infty,$$
for any nontrivial $u_{0}\geq 0$, then we see that the condition \eqref{lem2_for1} cannot be satisfied for $p\leq p_{F}:=1+2/D$. Here and in the sequel, when $M$ is a unimodular Lie group, we will simplify the notation by writing $\L$ instead of $\L_{M}$.

Usually, since the heat kernel is tightly connected to the volume growth, we will demonstrate below results on unimodular Lie groups, where only two situations may occur for the volume growth: polynomial and exponential. We will also discuss the obtained results in more general settings in Section \ref{SEC:app}.

Let $\G$ be a connected unimodular Lie group, endowed with the Haar measure, and let $X=\{X_{1},\cdots,X_{k}\}$ be a H\"{o}rmander system of left invariant vector fields. Let $\rho(x,y)$ be the Carnot-Carath\'{e}odory distance $\G\times \G \ni (x,y) \mapsto \rho(x,y)$ associated with $X$. We denote by $\rho(x)$ the distance from the unit element of the group to $x\in \G$. Let $V(t)$ be the volume of the ball $B(x,t)$ centred at $x\in\G$ and of radius $t>0$ for this distance. In this case, since the left invariant vector fields on $\G$ are divergence free with respect to the (right) Haar measure, as a consequence of \eqref{sublap_subR} the sub-Laplacian associated to the Haar measure
has the form of \enquote{sum of squares}, that is,
$$\L:=\sum_{i=1}^{k}X_{i}^{2}.$$
Then, on $\G$, the Cauchy problem \eqref{Cauchy_intro_Riem} becomes
\begin{equation}\label{Cauchy_intro}
\left\{
                \begin{array}{ll}
                  u_{t}-\L u=f(u), \;x\in \G, \;t>0,\\
                  u(0,x)=u_{0}(x), \;x\in \G,\\
                                  \end{array}
              \right.
\end{equation}
for $u_{0}\geq 0$.

Recall that we have $V(t)\simeq t^{d}$ for $t\in (0,1)$, where $d=d(\G,X)\in \mathbb{N}$ is the local dimension. In the case $t\geq1$, as we mentioned above, only two situations may occur, independently of the choice of $X$ (see e.g. \cite{CRT01} or \cite{Gui73}): either $\G$ has polynomial volume growth of order $D$, which means that there exists the global dimension $D=D(\G)\in \mathbb{N}_{0}$ (i.e. $\mathbb{N}\cup \{0\}$) such that $V(t)\simeq t^{D}$, $t\geq 1$, or $\G$ has exponential volume growth, that is, there exist positive constants $c_{1}$, $C_{1}$, $c_{2}$ and $C_{2}$ such that $c_{1}e^{c_{2} t}\leq V(t)\leq C_{1}e^{C_{2} t}$ for $t\geq1$. Note that (see e.g. \cite[Page 285]{CRT01}) the dimension $D$ at infinity depends only on the group $\G$ but not on the system $X$. Let us also recall that the closed subgroups of nilpotent Lie groups, connected Type $R$ Lie groups, motion groups, the Mautner group and compact groups are all examples of polynomial growth groups (see e.g. \cite[Section 1.5]{Sch93}). For examples of the unimodular Lie groups with exponential volume growth we can refer e.g. \cite[Section 2]{CM96} and references therein.

Let us now state the main results on $\G$:
\begin{thm}\label{main_thm} Let $\G$ be a connected unimodular Lie group with polynomial volume growth of order $D$ and let $1<p<\infty$.
\begin{enumerate}[label=(\roman*)]
\item Let $p<p_{F}=1+2/D$. Let $f:[0,\infty)\to [0,\infty)$ be a locally integrable function such that $f(u)\geq K_{2}u^{p}$ for some positive constant $K_{2}>0$. Then the differential inequality
\begin{equation}\label{heat_eq_thm_ineq}
u_{t}-\L u\geq f(u)
\end{equation}
does not admit any nontrivial distributional solution $u\geq 0$ in $(0,\infty)\times\G$.
\item Let $p=p_{F}=1+2/D$. Let $f:[0,\infty)\to [0,\infty)$ be a locally integrable function such that $f(u)\geq K_{2}u^{p}$ for some positive constant $K_{2}>0$. Then the equation
\begin{equation}\label{heat_eq_thm}
u_{t}-\L u= f(u)
\end{equation}
does not admit any nontrivial distributional solution $u\geq 0$ in $(0,\infty)\times\G$.
\item Let $p>p_{F}=1+2/D$. Let $f:[0,\infty)\to[0,\infty)$ be a continuous increasing function such that $f(u)\leq K_{1}u^{p}$ for some positive constant $K_{1}>0$. Then, for any $1\leq q<\infty$ the Cauchy problem \eqref{Cauchy_intro} has a global, classical solution for some positive $u_{0}\in L^{q}(\G)$.
\end{enumerate}
\end{thm}
\begin{rem} By a distributional solution, we mean in Parts (i) and (ii) of Theorem \ref{main_thm} a function $u\in L^{p}_{loc}(Q)$ which satisfies \eqref{heat_eq_thm_ineq} and \eqref{heat_eq_thm} in $\mathcal{D}^{\prime}(Q)$, respectively, where $Q:=(0,\infty)\times\G$.
\end{rem}
\begin{thm}\label{main_thm_exp} Let $\G$ be a connected unimodular Lie group with exponential volume growth and let $1<p<\infty$. Let $f:[0,\infty)\to[0,\infty)$ be a continuous increasing function such that $f(u)\leq K_{1}u^{p}$ for some positive constant $K_{1}>0$. Then, for any $1\leq q<\infty$ the Cauchy problem \eqref{Cauchy_intro} has a global, classical solution for some positive $u_{0}\in L^{q}(\G)$.
\end{thm}
\begin{rem} Consider problem \eqref{Cauchy_intro} with $f(u)\simeq u^{p}$. Then, combining these Theorems \ref{main_thm} and \ref{main_thm_exp}, one comes to the following interesting conclusion: In the case $D=0$ (e.g. when $\G$ is a compact group), that is, the case when the volume growth at infinity is constant, we see from Theorem \ref{main_thm} that the Cauchy problem \eqref{Cauchy_intro} does not admit any nontrivial distributional solution $u\geq 0$ in $(0,\infty)\times\G$ for $1<p<\infty$. In the case of polynomial volume growth, there exists a global, classical solution of \eqref{Cauchy_intro} for $p>p_{F}=1+2/D$ and some positive $u_{0}\in L^{q}(\G)$ with $1\leq q<\infty$. When $\G$ has exponential volume growth, then the Cauchy problem \eqref{Cauchy_intro} has a global, classical solution for $p>1$ and some positive $u_{0}\in L^{q}(\G)$ with $1\leq q<\infty$ by Theorem \ref{main_thm_exp}. For compact Lie groups the same kind of phenomenon (blow-up in
finite time for all $p>1$ under suitable sign assumptions for the Cauchy data) has
been recently proved also for the semilinear wave and damped wave equations, see \cite{Pal21a, Pal21b}.
\end{rem}
Concerning the existence in Part (iii) of Theorem \ref{main_thm}, we actually have the following much stronger property:

\begin{thm}\label{main_thm3} Let $\G$ be a connected unimodular Lie group with polynomial volume growth of order $D$. Consider problem \eqref{Cauchy_intro} with $p_{F}<p<\infty$, $u_{0}\in L^{q}(\G)$ with $1\leq q<\infty$, and let $\gamma>0$. Let $f:[0,\infty)\to[0,\infty)$ be a continuous increasing function such that $f(u)\leq K_{1}u^{p}$ for some positive constant $K_{1}>0$. There exists $\varepsilon=\varepsilon(\gamma)>0$ such that, if
\begin{equation}\label{proof_part2_for1}
0\leq u_{0}(x)\leq \varepsilon h_{\gamma}(x),\;x\in \G,
\end{equation} then there exists a non-negative continuous curve $u:[0,\infty)\rightarrow L^{q}(\G)$ which is a global solution to \eqref{Cauchy_intro} with initial value $u_{0}$. Moreover, we have
\begin{equation}\label{proof_part2_for2}
u(t,x)\leq Ch_{t+\gamma}(x),\;x\in \G, \;t\in (0,\infty),
\end{equation}
for some $C=C(\gamma)>0$.
\end{thm}
Similarly, concerning Theorem \ref{main_thm_exp}, we have the following stronger property:
\begin{thm}\label{main_thm3_exp} Let $\G$ be a connected unimodular Lie group with exponential volume growth. Consider problem \eqref{Cauchy_intro} with $1<p<\infty$, $u_{0}\in L^{q}(\G)$ with $1\leq q<\infty$, and let $\gamma>0$. Let $f:[0,\infty)\to[0,\infty)$ be a continuous increasing function such that $f(u)\leq K_{1}u^{p}$ for some positive constant $K_{1}>0$. There exists $\varepsilon=\varepsilon(\gamma)>0$ such that, if
\begin{equation}\label{proof_part2_for1_exp}
0\leq u_{0}(x)\leq \varepsilon h_{\gamma}(x),\;x\in \G,
\end{equation} then there exists a non-negative continuous curve $u:[0,\infty)\rightarrow L^{q}(\G)$ which is a global solution to \eqref{Cauchy_intro} with initial value $u_{0}$. Moreover, we have
\begin{equation}\label{proof_part2_for2_exp}
u(t,x)\leq Ch_{t+\gamma}(x),\;x\in \G, \;t\in (0,\infty),
\end{equation}
for some $C=C(\gamma)>0$.
\end{thm}

In the abelian case $\G=(\Rn,+)$, the phenomenon of finite time blow up was first considered by H. Fujita in \cite{F66}, where in the proof Gaussian test functions depending only on $x$ (given by the heat kernel with $t$ as a parameter) were involved, hence requiring more regularity of the solutions in time. Since then, many people devoted themselves to this problem. For example, we refer to \cite{MP01} and \cite{QP07}, where the proof is based on rescalings of a simple, compactly supported test-function, depending on $x$ and $t$. A related proof can be found in \cite{BP85}, where the test-functions were obtained by solving an adjoint problem. We also refer to \cite{HM04, LP76, Q91, SW97, W80, W81} and the references therein for the case of $\G=(\Rn,+)$, as well as \cite{JKS16}, \cite{GP21} on the Heisenberg group and \cite{Pas98} on stratified Lie groups. In the case of stratified Lie groups the blow-up result has been studied also in \cite{GP19}. There is huge literature on such Euclidean problems that we do not even attempt to review here. 

For a comparison principle for weak solutions of $p$-Laplacian heat equation in a bounded domain, we refer to recent works \cite{LZZ18} when $\G=(\Rn,+)$, and to \cite{RS18} and \cite{RY20} when $\G$ is a graded Lie group, the latter two also allowing more general hypoelliptic differential operators (Rockland operators).

The paper is structured as follows. In Section \ref{SEC:proof} we give the proof of the main results. Finally, these results are discussed on more general settings in Section \ref{SEC:app}.

\medskip
The authors would like to thank Michinori Ishiwata from Osaka University for drawing our attention to the important literature on the subject. The authors also would like to thank Tommaso Bruno from Ghent University for his comments on the heat kernel estimates in the exponential volume growth case.

\section{Proofs}
\label{SEC:proof}
\begin{proof}[Proof of Theorem \ref{lem2}] Here, the argument required for the extension to the sub-Riemannian manifold $M$ is virtually the same as employed by Weissler in \cite[Theorem 3]{W81}, see also \cite[Section 20]{QP07} for more details and other arguments. So, as in \cite{W81} we study \eqref{Cauchy_intro_Riem} via the following integral equation:
$$u(t)=e^{t\L_{M}}u_{0}+\int_{0}^{t}e^{(t-s)\L_{M}}f(u(s))ds.$$
Denote
$$\omega(t):=\left(1-K_{1}(p-1)\int_{0}^{t}\|e^{s\L_{M}}u_{0}\|^{p-1}_{L^{\infty}(M)}ds\right)^{-\frac{1}{p-1}}.$$
Note that $\omega(0)=1$ and $\omega'(t)=K_{1}\|e^{t\L_{M}}u_{0}\|^{p-1}_{L^{\infty}(M)}(\omega(t))^{p}$. Then we have
\begin{equation}\label{eq_w(t)_1}
\omega(t)=1+K_{1}\int_{0}^{t}\|e^{s\L_{M}}u_{0}\|^{p-1}_{L^{\infty}(M)}(\omega(s))^{p}ds.
\end{equation}
Let $u:[0,\infty)\rightarrow L^{q}(M)$ be a continuous curve with $1\leq q<\infty$ and $e^{t\L_{M}}u_{0}\leq u(t)\leq \omega(t)e^{t\L_{M}}u_{0}$ for all $t\geq 0$. Noting this, if we denote
$$\mathfrak{F}u(t):=e^{t\L_{M}}u_{0}+\int_{0}^{t}e^{(t-s)\L_{M}}f(u(s))ds,$$
then using $f(u)\leq Ku^{p}$ and the positivity of the heat kernel (see Introduction for the references), we get
\begin{equation}\label{eq_w(t)_2}
\begin{split}
\mathfrak{F}u(t)&\leq e^{t\L_{M}}u_{0}+K_{1}\int_{0}^{t}e^{(t-s)\L_{M}}(u(s))^{p}ds
\\&\leq e^{t\L_{M}}u_{0}+K_{1}\int_{0}^{t}e^{(t-s)\L_{M}}(e^{s\L_{M}}u_{0})^{p}(\omega(s))^{p}ds\\&
\leq e^{t\L_{M}}u_{0}+K_{1}\int_{0}^{t}e^{(t-s)\L_{M}}(e^{s\L_{M}}u_{0})\|e^{s\L_{M}}u_{0}\|^{p-1}_{L^{\infty}(M)}(\omega(s))^{p}ds\\&
=e^{t\L_{M}}u_{0}\left(1+K_{1}\int_{0}^{t}\|e^{s\L_{M}}u_{0}\|^{p-1}_{L^{\infty}(\G)}(\omega(s))^{p}ds\right),
\end{split}
\end{equation}
which implies with \eqref{eq_w(t)_1} that $e^{t\L_{M}}u_{0}\leq \mathfrak{F}u(t)\leq \omega(t)e^{t\L_{M}}u_{0}$ for all $t\geq0$.

Now we take the sequence of functions $\{v_{k}(t)\}_{k=0}^{\infty}$ such that $v_{0}(t)=e^{t\L_{M}}u_{0}$ and $v_{k+1}(t)=\mathfrak{F}v_{k}(t)$, and show that this sequence converges to the desired solution. Note that since $e^{t\L_{M}}u_{0}\leq v_{0}(t)\leq \omega(t)e^{t\L_{M}}u_{0}$ then by induction and discussing as in \eqref{eq_w(t)_2}, we get $e^{t\L_{M}}u_{0}\leq v_{k}(t)\leq \omega(t)e^{t\L_{M}}u_{0}$ for each $k$. We also note that since $v_{0}(t)\leq v_{1}(t)$ for all $t\geq0$, and $v_{k}(t)\leq v_{k+1}(t)\Rightarrow \mathfrak{F} v_{k}(t)\leq \mathfrak{F}v_{k+1}(t)$ for all $t\geq0$ by monotonicity of function $f$, then by induction one obtains $v_{k}(t)\leq v_{k+1}(t)$ for all $t\geq0$. Then, the dominated convergence theorem implies that $v_{k}(t)$ converge in $L^{q}(M)$ to a function which we call $u(t)$. This with the fact that $e^{t\L_{M}}u_{0}\leq v_{k}(t)\leq \omega(t)e^{t\L_{M}}u_{0}$ for each $k$ gives $e^{t\L_{M}}u_{0}\leq u(t)\leq \omega(t)e^{t\L_{M}}u_{0}$ for all $t\geq0$.

Now we need to prove that the function $u(t)$ is a global, classical solution of \eqref{Cauchy_intro_Riem}. Note that since we have $v_{k}(t)\leq\omega(t)e^{t\L_{M}}u_{0}$ and $f(u)\leq K_{1}u^{p}$, then the functions $s\mapsto e^{(t-s)\L_{M}}f(v_{k}(s))$ are dominated by
\begin{equation*}
\begin{split}
e^{(t-s)\L_{M}}f(v_{k}(s))&\leq K_{1}e^{(t-s)\L_{M}}(v_{k}(s))^{p}\\&\leq
K_{1}e^{(t-s)\L_{M}}(\omega(s)e^{s\L_{M}}u_{0})^{p}\\&\leq K_{1}e^{t\L_{M}}u_{0}\|e^{s\L_{M}}u_{0}\|_{L^{\infty}(M)}^{p-1}(\omega(s))^{p}
\end{split}
\end{equation*}
in $L^{1}(0,t;L^{q}(M))$. These functions converge for every $0<s<t$ to $e^{(t-s)\L_{M}}f(u(s))$ monotonically in $L^{q}(M)$ since the dominating function is in $L^{q}(M)$ for every $0<s<t$ and the fact that $f$ is a continuous function. Then, the dominated convergence theorem for $L^{q}$-valued functions implies that
$$\lim_{k\rightarrow\infty}\int_{0}^{t}e^{(t-s)\L_{M}}f(v_{k}(s))ds=\int_{0}^{t}e^{(t-s)\L_{M}}f(u(s))ds,$$
which gives
$$u(t)=\lim_{k\rightarrow \infty} v_{k+1}(t)=\lim_{k\rightarrow \infty} \mathfrak{F}v_{k}(t)=\mathfrak{F} u(t),$$
which means that $u(t)$ is a global solution of \eqref{Cauchy_intro_Riem}. Continuity of $u(t)$ in $L^{q}(M)$ easily follows by standard arguments.

The proof is complete.
\end{proof}
Before giving the proof of the main results on unimodular groups, let us briefly recall some necessary notations and some facts from \cite{VCS92}.

Recall that the heat kernel $(t,x)\mapsto h_{t}(x)$ is a positive fundamental solution of $u_{t}-\L u=0$. This heat kernel satisfies the following (see e.g. \cite[Section I.3, Page 5]{VCS92} or \cite[Section 2.1.1, Page 295]{CRT01}) property:
\begin{equation}\label{h_norm_1}
\|h_{t}\|_{L^{1}(\G)}=1, \;\forall t>0.
\end{equation}
Let $\rho(x,y)$ be the Carnot-Carath\'{e}odory distance $\G\times \G \ni (x,y) \mapsto \rho(x,y)$ associated with the H\"{o}rmander system of left invariant vector fields $X$. We also recall that $\rho$ is symmetric and satisfies the triangle inequality (see e.g. \cite[Section III.4, Page 39]{VCS92}).

We will also use the following fundamental properties:
\begin{thm}\cite[VIII.2.9 Theorem]{VCS92}\label{est_heat_thm} Let $\G$ be a connected unimodular Lie group with polynomial volume growth. Then there exist positive constants $C_{1}>0$ and $C_{2}>0$ such that
\begin{equation}\label{est_heat_above}
C_{1} V(\sqrt{t})^{-1}\exp\left(-C_{2}\frac{(\rho(x))^{2}}{t}\right)\leq h_{t}(x)\leq C_{2} V(\sqrt{t})^{-1}\exp\left(-C_{1}\frac{(\rho(x))^{2}}{t}\right),
\end{equation}
for all $t>0$ and $x\in \G$.
\end{thm}
\begin{thm}\cite[VIII.4.3 Theorem]{VCS92}\label{est_heat_thm_exp} Let $\G$ be a connected unimodular Lie group with exponential volume growth. Then for every $n\geq d$ and $\varepsilon>0$, there exist $C_{n,\varepsilon}$ such that
\begin{equation}\label{est_heat_above_exp}
|h_{t}(x)|\leq C_{n,\varepsilon} t^{-\frac{n}{2}}\exp\left(-\frac{(\rho(x))^{2}}{(4+\varepsilon)t}\right),
\end{equation}
for all $t>0$ and $x\in \G$, where $d$ is the local dimension of $\G$. 
\end{thm}
In order to prove Parts (i) and (ii) of Theorem \ref{main_thm}, it is enough to prove the following result:
\begin{thm}\label{main_thm2} Let $\G$ be a connected unimodular Lie group with polynomial volume growth of order $D$ and $1<p<\infty$. Let $f:[0,\infty)\to [0,\infty)$ be a locally integrable function such that $f(u)\geq K_{2}u^{p}$ for some positive constant $K_{2}>0$. Let $u_{0}:\G\rightarrow [0,\infty]$ be measurable, such that $u_{0}>0$ in a set of positive measure.
\begin{enumerate}[label=(\roman*)]
\item If $p<p_{F}=1+2/D$, then there is no nonnegative measurable global solution $u:[0,\infty)\times\G\rightarrow [0,\infty]$ to the integral inequality
\begin{equation}\label{int_form_ineq}
u(t,x)\geq \int_{\G}h_{t}(y^{-1}x)u_{0}(y)dy+\int_{0}^{t}\int_{\G}h_{t-s}(y^{-1}x)f(u(s,y))dyds
\end{equation}
such that $u(t,x)<\infty$ for a.e. $(t,x)\in (0,\infty)\times \G$.
\item If $p=p_{F}=1+2/D$, then there is no nonnegative measurable global solution $u:[0,\infty)\times\G\rightarrow [0,\infty]$ to the integral equation
\begin{equation}\label{int_form}
u(t,x)=\int_{\G}h_{t}(y^{-1}x)u_{0}(y)dy+\int_{0}^{t}\int_{\G}h_{t-s}(y^{-1}x)f(u(s,y))dyds
\end{equation}
such that $u(t,x)<\infty$ for a.e. $(t,x)\in (0,\infty)\times \G$.
\end{enumerate}
\end{thm}

Now let us show the following lemma on sub-Riemannian manifold $M$, which we will use in the proof of Theorem \ref{main_thm2} when $M$ is a connected unimodular Lie group:
\begin{lem}\label{lem1} Let $M$ be a sub-Riemannian manifold with $\int_{M}h_{t}(x,y)d\mu_{y}\leq1$ for all $x\in M$ and $t>0$. Let $1<p<\infty$ and $T>0$. Let $f:[0,\infty)\to [0,\infty)$ be a locally integrable function such that $f(u)\geq K_{2}u^{p}$ for some positive constant $K_{2}>0$. Let $\vartheta_{0}:M\rightarrow [0,\infty]$ and $\vartheta:[0,T]\times M \rightarrow [0,\infty]$ be measurable and satisfy
\begin{equation}\label{lem_for1}
\vartheta(t)\geq e^{t\L_{M}}\vartheta_{0}+\int_{0}^{t}e^{(t-s)\L_{M}}f(\vartheta(s))ds
\end{equation}
a.e. in $Q^{M}_{T}:=[0,T]\times M$. Assume that $\vartheta(t,x)<\infty$ for a.e. $(t,x)\in Q^{M}_{T}$. Then we have
\begin{equation}\label{lem_for2}
t^{\frac{1}{p-1}}\|e^{t\L_{M}}\vartheta_{0}\|_{L^{\infty}(M)}\leq A_{p}:=(K_{2}(p-1))^{-\frac{1}{p-1}}
\end{equation}
for all $t\in (0,T]$.
\end{lem}
\begin{proof}[Proof of Lemma \ref{lem1}] Note that by virtue of Fubini's theorem for nonnegative measurable functions we will use operations such as interchange of integrals and moving of $e^{-t\L_{M}}$ inside integrals in the proof of this lemma. We notice that
\begin{equation}\label{proof_lem_for1}
e^{t\L_{M}}F=e^{(t-s)\L_{M}}e^{s\L_{M}}F
\end{equation}
for all $0<s<t$ and any measurable $F:M\rightarrow [0,\infty]$. Also, we obtain from Jensen's inequality and $\int_{M}h_{t}(x,y)d\mu_{y}\leq1$ for all $x\in M$ and $t>0$ that
\begin{equation}\label{proof_lem_for2}
e^{t\L_{M}}F^{p}\geq (e^{t\L_{M}}F)^{p}
\end{equation}
for all measurable $F:M\rightarrow [0,\infty]$. Now, by redefining $u$ on a null set, one may assume that \eqref{lem_for1} actually holds everywhere in $(0,T)\times M$. By assumption, we have $\vartheta(\tau,\cdot)<\infty$ a.e. in $M$ for a.e. $\tau\in (0,T)$. Let us fix such $\tau$ and denote $M_{\tau}:=\{x\in  M:\vartheta(\tau,x)<\infty\}$. Then, \eqref{lem_for1} with $f(u)\geq K_{2}u^{p}$, \eqref{proof_lem_for1} and \eqref{proof_lem_for2} imply for $t\in [0,\tau]$ that
\begin{equation}\label{proof_lem_for3}
\begin{split}
e^{(\tau-t)\L_{M}}\vartheta(t)&\geq e^{\tau \L_{M}}\vartheta_{0}+K_{2}\int_{0}^{t}e^{(\tau-s)\L_{M}}(\vartheta(s))^{p}ds
\\&\geq e^{\tau \L_{M}}\vartheta_{0}+K_{2}\int_{0}^{t}(e^{(\tau-s)\L_{M}}\vartheta(s))^{p}ds=:g(t,\cdot),
\end{split}
\end{equation}
where we have also used that the heat kernel is positive (see Introduction for the references), so that the integration with it preserves inequalities between non-negative functions.
Here, from the second inequality in \eqref{proof_lem_for3} noting \eqref{lem_for1} and $f(u)\geq K_{2}u^{p}$, we get
\begin{equation}\label{proof_lem_for4}
\begin{split}
g(\tau, \cdot)&\leq e^{\tau\L_{M}}\vartheta_{0}+K_{2}\int_{0}^{\tau}e^{(\tau-s)\L_{M}}(\vartheta(s))^{p}ds
\\& \leq e^{\tau\L_{M}}\vartheta_{0}+\int_{0}^{\tau}e^{(\tau-s)\L_{M}}f(\vartheta(s))ds
\\& \leq  \vartheta(\tau,\cdot),
\end{split}
\end{equation}
and so $g(t,x)<\infty$ for all $(t,x)\in M_{\tau}\times[0,\tau]$. Fixing $x\in M_{\tau}$, we see that the function $\phi(t):=g(t,x)$ is absolutely continuous on $[0,\tau]$, and that \eqref{proof_lem_for3} implies
\begin{equation}\label{proof_lem_for5}
\phi'(t)=K_{2}(e^{(\tau-t)\L_{M}}\vartheta(t))^{p}(x)\geq K_{2}(\phi(t))^{p}
\end{equation}
for a.e. $t\in[0,\tau]$. For fixed $x\in M_{\tau}$ we have $\phi(t)=g(t,x)\geq (e^{\tau \L_{M}}\vartheta_{0})(x)>0$ by the definition of the function $g$ in \eqref{proof_lem_for3}, therefore, we can rewrite \eqref{proof_lem_for5} as $[\phi^{1-p}]'\leq -K_{2}(p-1)$. By integrating this inequality over $[0,\tau]$, we obtain
\begin{equation}\label{proof_lem_for6}
[(e^{\tau\L_{M}}\vartheta_{0})(x)]^{1-p}=\phi^{1-p}(0)\geq \phi^{1-p}(\tau)+K_{2}(p-1)\tau\geq K_{2}(p-1)\tau,
\end{equation}
which implies $\tau^{1/(p-1)}\|e^{\tau\L_{M}}\vartheta_{0}\|_{L^{\infty}(M)}\leq (K_{2}(p-1))^{-1/(p-1)}$. In particular, this means that $e^{t\L_{M}}\vartheta_{0}\in L^{\infty}( M)$ for a.e. $t\in (0,T)$. Since we know that $t\mapsto\|e^{t\L_{M}}\upsilon\|_{L^{\infty}(M)}$ is continuous for $\upsilon\in L^{\infty}(M)$ and $t>0$, then \eqref{proof_lem_for1} yields that the function
$$t\mapsto t^{\frac{1}{p-1}}\|e^{t\L_{M}}\vartheta_{0}\|_{L^{\infty}(M)}$$
is continuous in $(0,T)$, hence \eqref{lem_for2}.
\end{proof}
\begin{cor}\label{cor1} Let $M$ be a sub-Riemannian manifold with $\int_{M}h_{t}(x,y)d\mu_{y}\leq1$ for all $x\in M$ and $t>0$. Let $1<p<\infty$ and $T>0$. Let $f:[0,\infty)\to [0,\infty)$ be a locally integrable function such that $f(u)\geq K_{2}u^{p}$ for some positive constant $K_{2}>0$. Let $\vartheta_{0}: M\rightarrow [0,\infty]$ and $\vartheta:[0,T]\times M\rightarrow [0,\infty]$ be measurable and satisfy
$$
\vartheta(t)= e^{t\L_{M}}\vartheta_{0}+\int_{0}^{t}e^{(t-s)\L_{M}}f(\vartheta(s))ds
$$
a.e. in $Q^{M}_{T}=[0,T]\times M$.
Then we have
$$\|t^{\frac{1}{p-1}}e^{t\L_{M}}\vartheta(\tau)\|_{L^{\infty}( M)}\leq A_{p}:=(K_{2}(p-1))^{-\frac{1}{p-1}}$$
for all $t\in (0,T-\tau]$ and a.e. $\tau\in (0,T)$.
\end{cor}
\begin{proof}[Proof of Corollary \ref{cor1}] Denote $\widetilde{\vartheta}(t):=\vartheta(t+\tau)$. Then, by \eqref{proof_lem_for1} and Fubini's theorem we have for a.e. $\tau\in (0,T)$ and a.e. $t\in (\tau,T)$ that
\begin{equation*}
\begin{split}
\widetilde{\vartheta}(t)&= e^{(t+\tau)\L_{M}}\vartheta_{0}+\int_{0}^{t+\tau}e^{(t+\tau-s)\L_{M}}f(\vartheta(s))ds
\\&=e^{t\L_{M}}e^{\tau\L_{M}}\vartheta_{0}+\int_{0}^{\tau}e^{t\L_{M}}e^{(\tau-s)\L_{M}}f(\vartheta(s))ds+
\int_{\tau}^{t+\tau}e^{(t+\tau-s)\L_{M}}f(\vartheta(s))ds\\
&=e^{t\L_{M}}\left(e^{\tau\L_{M}}\vartheta_{0}+\int_{0}^{\tau}e^{(\tau-s)\L_{M}}f(\vartheta(s))ds\right)+\int_{0}^{t}e^{(t-s)\L_{M}}f(\widetilde{\vartheta}
(s))ds\\&
=e^{t\L_{M}}\vartheta({\tau})+\int_{0}^{t}e^{(t-s)\L_{M}}f(\widetilde{\vartheta}(s))ds.
\end{split}
\end{equation*}
Hence, Lemma \ref{lem1} with $\vartheta_{0}$ replaced by $\vartheta(\tau)$ and $T$ replaced by $T-\tau$ for a.e. $\tau\in (0,T)$ completes the proof.
\end{proof}
Now we are ready to prove Theorem \ref{main_thm2}.
\begin{proof}[Proof of Theorem \ref{main_thm2}] i) By way of contradiction let us assume that there exists a global solution of \eqref{int_form_ineq}. Then, by Lemma \ref{lem1} when $M=\G$ (where we have $\|h_{t}\|_{L^{1}(\G)}=1$ by \eqref{h_norm_1}) we get
\begin{equation}\label{proof_eq1}
|u_{0}\ast t^{\frac{1}{p-1}}h_{t}|\leq A_{p}.
\end{equation}
For $t>1$ from Theorem \ref{est_heat_thm} we get that
$$h_{t}(x)\geq C_{1}t^{-\frac{D}{2}}\exp\left(-C_{2}\frac{(\rho(x))^{2}}{t}\right)$$
for some positive constants $C_{1}$ and $C_{2}$. Then, using this for a given measurable function $v:\G\rightarrow [0,\infty]$, we have
\begin{equation}\label{proof_eq2}
\lim_{t\rightarrow \infty}v\ast t^{\frac{D}{2}}h_{t}\geq C\|v\|_{L^{1}(\G)}
\end{equation}
pointwise in $\G$, where $\|v\|_{L^{1}(\G)}:=\infty$ if $v\notin L^{1}(\G)$. In the case $p<p_{F}$, \eqref{proof_eq1} gives $t^{D/2}\|u_{0}\ast h_{t} \|_{L^{\infty}(\G)}\rightarrow 0$ as $t\rightarrow \infty$ which contradicts \eqref{proof_eq2} with $v=u_{0}$.

ii) In the case $p=p_{F}$, again by way of contradiction we assume that there exists a global solution of \eqref{int_form}. We redefine $u$ on a null set, then assuming that \eqref{int_form} actually holds everywhere in $(0,\infty)\times \G$, we get
\begin{equation}\label{int_form_2}
u(t+t_{0})= u(t_{0})\ast h_{t} +\int_{0}^{t}f(u(s+t_{0}))\ast h_{t-s}ds,
\end{equation}
for all $t,t_{0}>0$. Note that Corollary \ref{cor1} when $M=\G$ (where we have $\|h_{t}\|_{L^{1}(\G)}=1$ by \eqref{h_norm_1}) and \eqref{proof_eq2} guarantee the existence of positive constant $C$ such that
\begin{equation}\label{proof_eq3}
\|u(\tau)\|_{L^{1}(\G)}\leq C
\end{equation}
for a.e. $\tau>0$. On the other hand, since the Carnot-Carath\'{e}odory distance satisfies the triangle inequality, we have $(\rho(x,z))^{2}/(4t)\leq ((\rho(x))^{2}+(\rho(z))^{2})/(2t)$, and by Theorem \ref{est_heat_thm} we obtain
\begin{equation*}
u(t,x)\geq (u_{0}\ast h_{t})(x)\geq C_{1}t^{-\frac{D}{2}}e^{-2C_{2}\frac{(\rho(x))^{2}}{t}}
\int_{\G}e^{-2C_{2}\frac{(\rho(z))^{2}}{t}}u_{0}(z)dz
\end{equation*}
for $t\geq 1$, and
\begin{equation*}
u(t,x)\geq (u_{0}\ast h_{t})(x)\geq C_{1}t^{-\frac{d}{2}}e^{-2C_{2}\frac{(\rho(x))^{2}}{t}}
\int_{\G}e^{-2C_{2}\frac{(\rho(z))^{2}}{t}}u_{0}(z)dz
\end{equation*}
for $0<t<1$, which imply with \eqref{est_heat_above} that
$$u(2C_{2}/C_{1},x)\geq C_{3}h_{1},\;x\in \G.$$
Using this, and the property $h_{t+s}=h_{t}\ast h_{s}$ for $s,t>0$ and \eqref{int_form_2}, we deduce that
\begin{equation}\label{proof_eq4}
u(s+2C_{2}/C_{1})\geq u(2C_{2}/C_{1})\ast h_{s}\geq C_{3}h_{1}\ast h_{s}=C_{3}h_{s+1}, \;s>0.
\end{equation}
Now, Theorem \ref{est_heat_thm}, $(p-1)D/2=1$ and \eqref{h_norm_1} imply that
\begin{equation}\label{proof_eq5}
\begin{split}
\|h^{p}_{s+1}\|_{L^{1}(\G)}&\geq C_{1}^{p}(s+1)^{-pD/2}\left(\frac{s+1}{p}\right)^{D/2}\left(\frac{s+1}{p}\right)^{-D/2}
\int_{\G}e^{-C_{2}p\frac{(\rho(x))^{2}}{s+1}}dx\\&
\geq C_{4}(s+1)^{-1}\|h_{\frac{C_{1}(s+1)}{C_{2}p}}\|_{L^{1}(\G)}= C_{5}(s+1)^{-1}
\end{split}
\end{equation}
for all $s>1$ and some $C_{4}, C_{5}>0$. As in \cite[Proposition 48.4]{QP07}, one can note that from \eqref{h_norm_1} and Fubini's theorem we have $e^{t\L}\psi \geq0$ and 
$$\|e^{t\L}\psi\|_{L^{1}(\G)}=\int_{\G}\int_{\G}\psi(\zeta)h_{t}(\zeta^{-1}\eta)d\zeta d\eta=\int_{\G}\psi(\zeta)\left(\int_{\G}h_{t}(\zeta^{-1}\eta) d\eta\right)d\zeta=\|\psi\|_{L^{1}(\G)}$$
for any $\psi\geq0$ . This calculation, \eqref{proof_eq5}, \eqref{int_form_2} with $t_{0}=2C_{2}/C_{1}$ and \eqref{proof_eq4} imply
\begin{equation*}
\begin{split}
\|u(t+2C_{2}/C_{1})\|_{L^{1}(\G)}&\geq \int_{0}^{t}\|f(u(s+2C_{2}/C_{1}))\ast h_{t-s}\|_{L^{1}(\G)}ds\\
&\geq K_{2}\int_{0}^{t}\|u^{p}(s+2C_{2}/C_{1})\ast h_{t-s}\|_{L^{1}(\G)}ds
\\& \geq K_{2}\int_{0}^{t}\|(C_{3}h_{s+1})^{p}\ast h_{t-s}\|_{L^{1}(\G)}ds\\&
= K_{2}C_{3}^{p}\int_{0}^{t}\|h^{p}_{s+1}\|_{L^{1}(\G)}ds\\&\geq K_{2}C_{3}^{p}\int_{1}^{t}\|h^{p}_{s+1}\|_{L^{1}(\G)}ds\\&
\geq K_{2}C_{3}^{p}C_{5}\int_{1}^{t}(s+1)^{-1}ds\rightarrow \infty
\end{split}
\end{equation*}
as $t\rightarrow \infty$, which contradicts \eqref{proof_eq3}.
\end{proof}
Now we prove Theorem \ref{main_thm3}.
\begin{proof}[Proof of Theorem \ref{main_thm3}] To prove Theorem \ref{main_thm3} we see by Theorem \ref{lem2} that it is enough to show \eqref{lem2_for1}. Since by the assumption \eqref{proof_part2_for1} we have $u_{0}(x)\leq \varepsilon h_{\gamma}(x)$ for all $x\in \G$, and noting that $\|h_{t}\|_{L^{1}(\G)}=1, \;\forall t>0$, and Theorem \ref{est_heat_thm}, we obtain
\begin{equation*}
\begin{split}
\int_{0}^{\infty}\|e^{s\L }u_{0}\|_{L^{\infty}(\G)}^{p-1}ds&=
\int_{0}^{\infty}\|u_{0}\ast h_{s}\|_{L^{\infty}(\G)}^{p-1}ds
\\&\leq\varepsilon\int_{0}^{\infty}\|h_{\gamma}\ast h_{s}\|_{L^{\infty}(\G)}^{p-1}ds\\&
=\varepsilon\int_{0}^{\infty}\|h_{s+\gamma}\|_{L^{\infty}(\G)}^{p-1}ds\\&
=\varepsilon\int_{\gamma}^{\infty}\|h_{s}\|_{L^{\infty}(\G)}^{p-1}ds\\&
< C \varepsilon \left(\int_{\min(\gamma,1)}^{1}s^{-\frac{d(p-1)}{2}}ds
+\int_{1}^{\infty}s^{-\frac{D(p-1)}{2}}ds\right)\\&
< \frac{1}{K_{1}(p-1)}
\end{split}
\end{equation*}
for small $\varepsilon>0$ since $D(p-1)/2>1$ and $\gamma>0$.
\end{proof}
Now let us prove Theorem \ref{main_thm3_exp}.
\begin{proof}[Proof of Theorem \ref{main_thm3_exp}] Actually, the proof of this theorem is similar to the proof of Theorem \ref{main_thm3}, we use Theorem \ref{est_heat_thm_exp} instead of Theorem \ref{est_heat_thm}.

By \eqref{proof_part2_for1_exp}, $\|h_{t}\|_{L^{1}(\G)}=1, \;\forall t>0$, and Theorem \ref{est_heat_thm_exp}, one has
\begin{equation}\label{exp_est_case2}
\begin{split}
\int_{0}^{\infty}\|e^{s\L }u_{0}\|_{L^{\infty}(\G)}^{p-1}ds&=
\int_{0}^{\infty}\|u_{0}\ast h_{s}\|_{L^{\infty}(\G)}^{p-1}ds
\\&\leq\varepsilon\int_{0}^{\infty}\|h_{\gamma}\ast h_{s}\|_{L^{\infty}(\G)}^{p-1}ds\\&
=\varepsilon\int_{0}^{\infty}\|h_{s+\gamma}\|_{L^{\infty}(\G)}^{p-1}ds\\&
< C \varepsilon \int_{0}^{\infty}(s+\gamma)^{-\frac{n(p-1)}{2}}ds
\end{split}
\end{equation}
for small $\varepsilon>0$ and for every $n\geq d$. So, letting $n\rightarrow \infty$ we observe that in this case the condition \eqref{lem2_for1} holds for $1<p<\infty$. 

Thus, Theorem \ref{lem2} concludes the proof.
\end{proof}
\section{The global well-posedness on sub-Riemannian manifolds}
\label{SEC:app}
In this section we discuss the obtained results on unimodular groups in more general settings, namely, on sub-Riemannian manifolds $M$. To have an analogue of Part (i) of Theorem \ref{main_thm} on $M$, we need to assume that the following estimate for the heat kernel from below holds on $M$ (see the proof of Part (i) of Theorem \ref{main_thm2}): assume that there exist constants $C_{1},C_{2}>0$ and $a\geq 0$ such that
\begin{equation}\label{est_heat_below_ineq_gener}
h_{t}(x,y)\geq C_{1} t^{-\frac{a}{2}}\exp\left(-C_{2}\frac{(\rho(x,y))^{2}}{t}\right),
\end{equation}
for all $t>1$ and $x,y\in M$. Therefore, we have
\begin{thm}\label{main_thm_gener1} Assume that \eqref{est_heat_below_ineq_gener} holds on $M$ for some $a \geq 0$ and that $\int_{M}h_{t}(x,y)d\mu_{y}\leq1$ for all $x\in M$ and $t>0$. Let $1<p<1+2/a$. Let $f:[0,\infty)\to [0,\infty)$ be a locally integrable function such that $f(u)\geq K_{2}u^{p}$ for some positive constant $K_{2}>0$. Then the differential inequality
\begin{equation}\label{heat_eq_thm_ineq_gener}
u_{t}-\L_{M} u\geq f(u)
\end{equation}
does not admit any nontrivial distributional solution $u\geq 0$ in $(0,\infty)\times M$.
\end{thm}
By the proof of Part (ii) of Theorem \ref{main_thm2}, we note that to obtain an analogue of Part (ii) of Theorem \ref{main_thm} on $M$ one needs the following properties:
\begin{enumerate}
\item $h_{t+s}(x,y)=\int_{M}h_{t}(x,z)h_{s}(z,y)d\mu_{z}$ for all $x,y\in M$ and $s,t>0$;
\item $\int_{M}h_{t}(x,y)d\mu_{y}=1, \;\forall x\in M,\;\forall t>0$;
\item There exist constants $\{C_{i}\}_{i=1}^{8}>0$ and $a>0$, $b\geq 0$ such that
\begin{equation}\label{est_heat_gen_less1}
C_{1} t^{-\frac{b}{2}}\exp\left(-C_{2}\frac{(\rho(x,y))^{2}}{t}\right)\leq h_{t}(x,y)\leq C_{3} t^{-\frac{b}{2}}\exp\left(-C_{4}\frac{(\rho(x,y))^{2}}{t}\right),
\end{equation}
for all $0<t<1$ and $x,y\in M$, and
\begin{equation}\label{est_heat_gen_great1}
C_{5} t^{-\frac{a}{2}}\exp\left(-C_{6}\frac{(\rho(x,y))^{2}}{t}\right)\leq h_{t}(x,y)\leq C_{7} t^{-\frac{a}{2}}\exp\left(-C_{8}\frac{(\rho(x,y))^{2}}{t}\right),
\end{equation}
for all $t\geq 1$ and $x,y\in M$.
\end{enumerate}
Note that we always have the above property (1) on $M$ whenever the heat kernel exists.

Therefore, the following theorem can be an analogue of Part (ii) of Theorem \ref{main_thm} on $M$:
\begin{thm}\label{main_thm_gener2} Assume that (2)-(3) hold on $M$ for some $a>0$ and $b\geq 0$. Let $p=1+2/a<\infty$. Let $f:[0,\infty)\to [0,\infty)$ be a locally integrable function such that $f(u)\geq K_{2}u^{p}$ for some positive constant $K_{2}>0$. Then the equation
\begin{equation}\label{heat_eq_thm_gener}
u_{t}-\L_{M} u= f(u)
\end{equation}
does not admit any nontrivial distributional solution $u\geq 0$ in $(0,\infty)\times M$.
\end{thm}
\begin{rem} By a distributional solution, we mean in Theorems \ref{main_thm_gener1} and \ref{main_thm_gener2} a function $u\in L^{p}_{loc}(Q^{M})$ which satisfies \eqref{heat_eq_thm_ineq_gener} and \eqref{heat_eq_thm_gener} in $\mathcal{D}^{\prime}(Q^{M})$, respectively, where $Q^{M}:=(0,\infty)\times M$.
\end{rem}
As for an analogue of Theorem \ref{main_thm3} on $M$, since we already have Theorem \ref{lem2} on $M$, we only need to check \eqref{lem2_for1}. For this, since we have used the estimate for the heat kernel from above in the proof of Theorem \ref{main_thm3}, then to obtain an analogue of Theorem \ref{main_thm3} on $M$ one needs to assume that the following estimates hold on $M$: there exist constants $C_{3},C_{4}, C_{7}, C_{8}>0$ and $a>0$, $b\geq0$ such that
\begin{equation}\label{est_heat_above_gener_less1}
h_{t}(x,y)\leq C_{3} t^{-\frac{b}{2}}\exp\left(-C_{4}\frac{(\rho(x,y))^{2}}{t}\right),
\end{equation}
for all $0<t< 1$ and $x,y\in M$, and
\begin{equation}\label{est_heat_above_gener}
h_{t}(x,y)\leq C_{7} t^{-\frac{a}{2}}\exp\left(-C_{8}\frac{(\rho(x,y))^{2}}{t}\right),
\end{equation}
for all $t\geq 1$ and $x,y\in M$.
Therefore, we have the following theorem on $M$:
\begin{thm}\label{main_thm3_gener} Assume that \eqref{est_heat_above_gener_less1} and \eqref{est_heat_above_gener} hold on $M$ for some $a>0$ and $b\geq 0$. Consider the problem \eqref{Cauchy_intro_Riem} with $1+2/a<p<\infty$. Let $u_{0}\in L^{q}(M)$ with $1\leq q<\infty$ and $\gamma>0$. Let $f:[0,\infty)\to[0,\infty)$ be a continuous increasing function such that $f(u)\leq K_{1}u^{p}$ for some positive constant $K_{1}>0$. There exists $\varepsilon=\varepsilon(\gamma)>0$ such that, for every $y\in M$ if
\begin{equation}\label{proof_part2_for1_gener}
0\leq u_{0}(x)\leq \varepsilon h_{\gamma}(x,y),\;x\in M,
\end{equation}
then there exists a non-negative continuous curve $u:[0,\infty)\rightarrow L^{q}(M)$ which is a global solution to \eqref{Cauchy_intro_Riem} with initial value $u_{0}$. Moreover, we have
\begin{equation}\label{proof_part2_for2_gener}
0\leq u(t,x)\leq Ch_{t+\gamma}(x,y),\;x,y\in M, \;t\in (0,\infty),
\end{equation}
for some $C=C(\gamma)>0$.
\end{thm}
In particular, Theorem \ref{main_thm3_gener} implies an analogue of Part (iii) of Theorem \ref{main_thm} on $M$:
\begin{thm}\label{main_thm4_gener} Assume that \eqref{est_heat_above_gener_less1} and \eqref{est_heat_above_gener} hold on $M$ for some $a>0$ and $b\geq 0$. Let $1+2/a<p<\infty$. Let $f:[0,\infty)\to[0,\infty)$ be a continuous increasing function such that $f(u)\leq K_{1}u^{p}$ for some positive constant $K_{1}>0$. Then, for any $1\leq q<\infty$ the Cauchy problem \eqref{Cauchy_intro_Riem} has a global, classical solution for some positive $u_{0}\in L^{q}(M)$.
\end{thm}
Now we give some examples. Let us first recall the following result from \cite{Sal10} (see also \cite{Gri91} and \cite{Sal92}) on weighted Riemannian manifolds, that is, complete non-compact Riemannian manifolds equipped with a measure $\mu(dy)=\sigma(y)v(dy)$, $0<\sigma\in C^{\infty}(M)$, and the associated weighted Laplacian $\L_{M}^{\sigma}:=\sigma^{-1}\;{\rm div}\;(\sigma\;{\rm grad})$:
\begin{thm}\cite[Theorem 3.1]{Sal10}\label{two_side_est_heat_Sal} Let $M$ be a weighted complete Riemannian manifold. Then the following three properties are equivalent:
\begin{itemize}
\item The parabolic Harnack inequality (PHI).
\item The two-sided heat kernel bound $((t, x, y)\in(0,\infty)\times M\times M)$:
\begin{equation}\label{two_sided_heat_est}
\frac{\widetilde{c_{1}}}{V(x,\sqrt{t})}e^{-\widetilde{C_{1}}\frac{(\rho(x,y))^{2}}{t}}
\leq h_{t}(x,y)\leq \frac{\widetilde{C_{2}}}{V(x,\sqrt{t})}e^{-\widetilde{c_{2}}\frac{(\rho(x,y))^{2}}{t}}.
\end{equation}
\item The conjunction of
\begin{itemize}
\item The volume doubling property
$$\forall x \in M, r>0, V(x,2r)\leq D V(x,r).$$
\item The Poincar\'{e} inequality $(\forall x \in M, r>0, B=B(x,r))$
$$\forall f\in \;{\rm Lip}\;(B),\;\int_{B}|f-f_{B}|^{2}d\mu \leq Pr^{2}\int_{B}|\nabla f|^{2}d\mu,$$
where $f_{B}$ is the mean of $f$ over $B$.
\end{itemize}
\end{itemize}
\end{thm}
\begin{rem} Note that a complete weighted manifold $M$ satisfies (PHI) if and only if the Riemannian product $\mathbb{R}\times M$ satisfies the elliptic
Harnack inequality (see \cite{HS01}).
\end{rem}
We refer to \cite[Section 3.2]{Sal10} for more details.

We note by the proof of Theorem \ref{lem2}, Lemma \ref{lem1} and Corollary \ref{cor1} that we also have Theorem \ref{lem2}, Lemma \ref{lem1} and Corollary \ref{cor1} on weighted Riemannian manifolds satisfying the two-sided heat kernel bound \eqref{two_sided_heat_est} (hence also on weighted Riemannian manifolds satisfying (PHI) by virtue of Theorem \ref{two_side_est_heat_Sal}) with the weighted Laplacian, since \eqref{two_sided_heat_est} also implies that we have the positivity of the heat kernel on such weighted Riemannian manifolds. Examples of such weighted Riemannian manifolds are complete Riemannian manifolds with non-negative Ricci curvature, convex domains in Euclidean space, complements of any convex domain, connected Lie groups with polynomial volume growth, Riemannian manifolds which cover a compact manifold with deck transformation group $\Gamma$, complete Riemannian manifolds $M$ and $N$ such that $M/G=N$, where $G$ is a group of isometries of $M$, the Euclidean space $\Rn$, $n\geq2$, with weight $(1+|x|^{2})^{\alpha/2}$ and $\alpha>-n$. Hence, they are also examples of weighted Riemannian manifolds satisfying (PHI) because of Theorem \ref{two_side_est_heat_Sal} (see e.g. \cite[Section 3.3]{Sal10}).

Since now we have Lemma \ref{lem1} and Corollary \ref{cor1} on weighted Riemannian manifolds satisfying the two-sided heat kernel bound \eqref{two_sided_heat_est} (hence also on weighted Riemannian manifolds satisfying (PHI) by virtue of Theorem \ref{two_side_est_heat_Sal}), then taking into account the above discussions for Theorems \ref{main_thm_gener1}, \ref{main_thm_gener2}, \ref{main_thm3_gener} and \ref{main_thm4_gener}, we obtain these Theorems \ref{main_thm_gener1}, \ref{main_thm_gener2}, \ref{main_thm3_gener}, \ref{main_thm4_gener} on weighted Riemannian manifolds satisfying the two-sided heat kernel bound \eqref{two_sided_heat_est} with the volume growth, such that ultimately an estimate for the heat kernel has to has a form as in \eqref{est_heat_below_ineq_gener} for Theorem \ref{main_thm_gener1}, \eqref{est_heat_gen_less1}-\eqref{est_heat_gen_great1} for Theorem \ref{main_thm_gener2}, and \eqref{est_heat_above_gener_less1}-\eqref{est_heat_above_gener} for Theorems \ref{main_thm3_gener}-\ref{main_thm4_gener}. Here, we want to note that the volume growth does not have to be polynomial, see for example Theorem \ref{main_thm3_exp}, if the volume growth is exponential but we have an estimate of the type \eqref{est_heat_above_exp}.

\end{document}